\documentclass[12pt]{amsart}
\usepackage{amsmath, amssymb, mathtools}
\usepackage{graphicx}
\usepackage{diagbox}
\usepackage{breqn}

\newcommand{\rr}{\mathbb{R}}

\newtheorem{Def}{Definition}
\newtheorem{Lem}{Lemma}
\newtheorem{Prop}{Proposition}
\newtheorem{Thm}{Theorem}
\newtheorem{Cor}{Corollary}

\newenvironment{Pf}{ Proof.}{\(\square\)}
\newtheorem{Exm}{Example}

\newtheorem{Con}{Open problem}

\title[On locally symmetric polynomial metrics...]{On locally symmetric polynomial metrics: Riemannian and Finslerian surfaces}
\author{Csaba Vincze}
\address{Institute of Mathematics, University of Debrecen, H-4002 Debrecen, P. O. Box 400, Hungary}
\email{csvincze@science.unideb.hu}
\author{M\'{a}rk Ol\'{a}h}
\address{Institute of Mathematics, University of Debrecen, H-4002 Debrecen, P. O. Box 400, Hungary \newline
\indent ELKH-DE Equations, Functions, Curves and their Applications Research Group}
\email{olah.mark@science.unideb.com}
\author{\'{A}bris Nagy}
\address{Institute of Mathematics, University of Debrecen, H-4002 Debrecen, P. O. Box 400, Hungary}
\email{abris.nagy@science.unideb.hu}
\begin{document}
\begin{abstract} In the paper we investigate locally symmetric polynomial metrics in special cases of Riemannian and Finslerian surfaces. The Riemannian case will be presented by a collection of basic results (regularity of second root metrics) and formulas up to Gauss curvature. In case of Finslerian surfaces we formulate necessary and sufficient conditions for a locally symmetric fourth root metric in 2D to be positive definite. They are given in terms of the coefficients of the polynomial metric to make checking the positive definiteness as simple and direct as possible. Explicit examples are also presented. The situation is more complicated in case of spaces of dimension more than two. Some necessary conditions and an explicit example are given for a positive definite locally symmetric polynomial metric in 3D. Computations are supported by the MAPLE mathematics software (LinearAlgebra).
\end{abstract}
\maketitle
\section*{Introduction}
Let $M$ be a differentiable manifold with local coordinates $u_1, \ldots, u_n$ (as index calculus is limited, we are going to use   lower indices for the coordinates to make polynomial formulas easier to read). The induced coordinate system of the tangent manifold $TM$ consists of the functions $x_1=u_1\circ \pi, \ldots, x_n=u_n\circ \pi$ and $y_1=du_1, \ldots, y_n=du_n$, where $\pi\colon TM\to  M$ is the canonical projection of the tangent vectors to the base points. A Finsler metric \cite{BSC} is a non-negative continuous function $F\colon TM\to \mathbb{R}$ satisfying the following conditions: 
\begin{itemize}
\item[(F1)] $F$ is smooth on the complement of the zero section, 
\item[(F2)] $F(tv)=tF(v)$ for all $t> 0$ (positive homogeneity of degree one),
\item[(F3)] the Hessian matrix 
$$g_{ij}=\frac{\partial^2 E}{\partial y_i\partial y_j}$$ 
\end{itemize}
of the energy function $\displaystyle{E=F^2/2}$ is positive definite at all nonzero elements 
$\displaystyle{v \in \pi^{-1}(U)\subset TM}$. 
Using Euler's theorem for homogeneous functions, the positive definiteness of the Hessian implies that
$$2E(v)=\sum_{i, j=1}^n v_iv_j g_{ij}(v)>0$$ 
for any nonzero element of the tangent manifold. Since $F$ is non-negative it also follows that
$F(v)\geq 0$ and $F(v)=0$ if and only if $v={\bf 0}$.  The Riemann--Finsler metric $g_{ij}$ makes each tangent space (except at the
origin) a Riemannian manifold with standard canonical objects such as the volume form
$$d\mu=\sqrt{\det g_{ij}} dy_1\wedge \ldots \wedge dy_n,$$
the Liouville vector field
$$C=y_1\frac{\partial}{\partial y_1}+\ldots +y_n\frac{\partial}{\partial y_n}$$
and the induced volume form on the indicatrix hypersurface $\partial K_p= F^{-1}(1)\cap T_pM$, where $p\in M$. The coordinate expression is
$$\mu=\sqrt{\det g_{ij}}\sum_{i=1}^n (-1)^{i-1} \frac{y_i}{F} dy_1\wedge \ldots \wedge dy_{i-1}\wedge dy_{i+1}\wedge \ldots \wedge dy_n.$$

\begin{Def}
Let $m=2l$ be a positive natural number, $l=1, 2, \ldots. $ A Finsler metric $F$ is called an $m$-th root metric if its $m$-th power $F^m$ is of class $C^{m}$ on the tangent manifold $TM$.
\end{Def}

Using that $F$ is positively homogeneous of degree one, its $m$-th power is positively homogeneous of degree $m$. If it is of class $C^m$ on the tangent manifold $TM$ (including the zero section), then its local form must be a polynomial\footnote{It can be seen by a simple induction: if a zero homogeneous function is continuous at the origin, then the homogeneity property implies that for any real number $t>0$,
$$f(tv)=f(v)\ \ \Rightarrow \ \ f({\bf 0})=\lim_{t\to 0^+}f(tv)=f(v),$$
i.e. the function is constant in the tangent spaces. Since partial differentiation with respect to the variables 
$y_1$, $\ldots$, $y_n$ decreases the degree of homogeneity, we can apply the inductive hypothesis to conclude the statement.} of degree $m$ in the variables $y_1$, $\ldots$, $y_n$ as follows:
\begin{equation}
\label{polgen}
F^m(x,y)=\sum_{i_1+\ldots+i_n=m} a_{i_1 \ldots i_n}(x) y_1^{i_1} \ldots y_n^{i_n}.
\end{equation}
Finsler metrics of the form (\ref{polgen}) have been introduced by Shimada \cite{Shim}. They are generalizations of the so-called Berwald-Mo\'{o}r metrics. The geometry of $m$-th root metrics and some special cases have been investigated by several authors such as  M. Matsumoto, K. Okubo, V. Balan, N. Brinzei, L. Tam\'{a}ssy, A. Tayebi and B. Najafi etc. in \cite{BB}, \cite{VB}, \cite{Brin}, \cite{MO}, \cite{TN} and \cite{Tam}.

\begin{Exm}
{\emph{Riemannian metrics are $2$nd root metrics, i.e. $m=2$.}} 
\end{Exm}

\begin{Def}{\emph{\cite{OKV}}}
$F$ is a locally symmetric $m$-th root metric if each point has a coordinate neighbourhood such that $F^m$ is a symmetric polynomial of degree $m$ in the variables $y_1$, $\ldots$, $y_n$ of the induced coordinate system on the tangent manifold.
\end{Def}

\begin{Exm}
{\emph{Taking an arbitrary $m$-th root metric $F$ we can construct a locally symmetric $m$-th root metric over a coordinate neighbourhood by the symmetrization process
$$F_{\textrm{sym}}^m(x, y_1, \ldots, y_n)=\frac{1}{n!}\sum_{\sigma}F^m(x, y_{\sigma(1)}, \ldots, y_{\sigma(n)}),$$
where $\sigma$ runs through the permutations of the indices.}}
\end{Exm}

Suppose that formula (\ref{polgen}) is a symmetric expression of $F^m(x,y)$ in the variables $y_1$, $\ldots$, $y_n$. Using the fundamental theorem of symmetric polynomials, we can write that
\begin{equation}
\label{charpol}
F^m(x,y)=P(s_1, \ldots, s_n),
\end{equation}
where 
$$s_1=y_1+\ldots +y_n, \ s_2=y_1y_2+\ldots+y_{n-1}y_n, \ \ldots, \ s_n=y_1 \ldots y_n$$
are the so-called elementary symmetric polynomials. The polynomial $P$ with coefficients depending on the position is called the \emph{local characteristic polynomial} of the locally symmetric $m$-th root metric. Using the homogeneity properties, the reduction of the number of the coefficients depending on the position is
\begin{equation}
\label{polspec}
F^m(x,y)=\sum_{j_1+2j_2+\ldots+nj_n=m} c_{j_1 \ldots j_n}(x) s_1^{j_1}\ldots s_n^{j_n}.
\end{equation}
In case of $m=2$ (second root metrics -- Riemannian case) and $n\geq 2$, 
$$P(s_1, \ldots, s_n)=c_{20}(x) s_1^2+c_{01}(x)s_2.$$
In case of $m=4$ (fourth root metrics) and $n=2$, $3$, $4$ we have
\begin{align*}
P(s_1, s_2)&=c_{40}(x)s_1^4+c_{21}(x)s_1^2s_2+c_{02}(x)s_2^2,\\
P(s_1, s_2, s_3)&=c_{400}(x)s_1^4+c_{210}(x)s_1^2s_2+c_{020}(x)s_2^2+c_{101}(x)s_1s_3,\\
P(s_1, s_2,s_3,s_4)&=c_{4000}(x)s_1^4+c_{2100}(x)s_1^2s_2+c_{0200}(x)s_2^2+c_{1010}(x)s_1s_3+c_{0001}(x)s_4,
\end{align*}
respectively. Locally symmetric fourth root metrics for $n\geq 5$ are similar up to the formal zero coefficients of the terms $s_5$, $\ldots$, $s_n$.
\begin{Cor} {\emph{\cite{OKV}}} A locally symmetric fourth root metric is locally determined by at most five components of its local characteristic polynomial.
\end{Cor}
 Let us introduce the following notations:
$$A(x,y):=F^m(x,y)=\sum_{i_1+\ldots+i_n=m} a_{i_1 \ldots i_n}(x)y_1^{i_1} \ldots y_n^{i_n}, \ \ A_i:=\frac{\partial A}{\partial y_i} \ \ \textrm{and} \ \ A_{ij}:=\frac{\partial^2 A}{\partial y_i \partial y_j}.$$

\begin{Lem} \label{lemma:1}
For any $m$-th root metric the Hessian $A_{ij}$ is positive definite at all nonzero elements $\displaystyle{v \in \pi^{-1}(U)\subset TM}$.
\end{Lem}

\begin{Pf}
Since $\displaystyle{F^m/2^l=E^l}$, we have that
\begin{equation}
\label{lemma:01}
A_{ij}/2^l=l(l-1)E^{l-2}\frac{\partial E}{\partial y_i}\frac{\partial E}{\partial y_j}+lE^{l-1}\frac{\partial^2 E}{\partial y_i\partial y_j},
\end{equation}
where $m=2l$. This means that the positive definiteness of the Hessian of the energy function implies the positive definiteness of the Hessian of the $m$-th power function. 
\end{Pf}

\begin{Lem} \label{lemma:2}
If the Hessian $A_{ij}$ of a smooth, pointwise polynomial function 
$$A(x,y)=\sum_{i_1+\ldots+i_n=m} a_{i_1 \ldots i_n}(x) y_1^{i_1} \ldots y_n^{i_n}$$
is positive definite at all nonzero elements $\displaystyle{v \in \pi^{-1}(U)\subset TM}$ then its $m$-th root is a Finsler metric.
\end{Lem}

\begin{Pf}
Let us choose a nonzero element $v\in T_pM$. Using Euler's theorem for homogeneous functions, the positive definiteness
$$0< \sum_{i, j=1}^n v_i v_j A_{ij}(v)=m(m-1)A(v)$$
implies that we can introduce the function $F(v)=\sqrt[m]{A(v)}>0$ and $F({\bf 0})=0$. It is continuous, non-negative and smooth on the complement of the zero section (see (F1)). The homogeneity property (F2) is automatically satisfied. Let $z\in T_pM$ be a nonzero element of the form\footnote{For such a decomposition 
$$t=\sum_{i=1}^n\frac{z_i}{2E(v)}\frac{\partial E}{\partial y_i}(v),$$
where $E=F^2/2$.} $z=w+tv$ such that                                       
$$\sum_{i=1}^n w_i\frac{\partial E}{\partial y_i}(v)=0.$$
We have
$$\sum_{i, j=1}^n(w_i+tv_i)(w_j+tv_j)\frac{\partial^2 E}{\partial y_i\partial y_j}(v)=$$
$$\sum_{i, j=1}^n w_i w_j\frac{\partial^2 E}{\partial y_i\partial y_j}(v)+2t \sum_{i=1}^n w^i\frac{\partial E}{\partial y_i}(v)+t^2 \sum_{i, j=1}^n v_i v_j \frac{\partial^2 E}{\partial y_i\partial y_j}(v)=$$
$$\sum_{i, j=1}^n w_i w_j\frac{\partial^2 E}{\partial y_i\partial y_j}(v)+2t^2E(v)$$
by the homogeneity of degree $2$ of the energy function $E=F^2/2$. Therefore
$$0\leq \sum_{i, j=1}^n w_i w_j A_{ij}(v)/2^l=lE^{l-1}(v) \sum_{i, j=1}^n w_i w_j\frac{\partial^2 E}{\partial y_i\partial y_j}(v)$$
implies that the Hessian of the energy function is positive definite even if $w={\bf 0}$.  
\end{Pf}

\subsection*{Some notes about the linear isometry group} The \emph{linear isometry group} of a Finsler metric at a point $p\in M$ consists of linear transformations $\varphi\colon T_pM\to T_pM$ such that 
$$F\circ \varphi (v)=F(v)\ \ (v\in T_pM) \ \ \Rightarrow \ \ E\circ \varphi (v)=E(v)\ \ (v\in T_pM).$$
Using the linearity of $\varphi$, second order partial differentiation of the energy function by the directional variables $y_1, \ldots, y_n$ shows that $\varphi$ is an isometry with respect to the Riemann--Finsler metric in the sense that for any non-zero element $v\in T_pM$
$$\sum_{i, j=1}^n \varphi_i(w)\varphi_j(z)g_{ij}\circ \varphi(v)=\sum_{i, j=1}^n w_i z_j g_{ij}(v)\ \ (w, z\in T_pM).$$
Therefore the linear isometry group is a closed and, consequently, compact subgroup in the orthogonal group with respect to the \emph{averaged Riemannian metric} 
$$\gamma_{ij}(p)=\int_{\partial K_p} g_{ij}\, \mu;$$
for more details about the applications of average processes in Finsler geometry see \cite{V11}. 
Suppose that formula (\ref{polgen}) is a symmetric expression of $F(x,y)$ in the variables $y_1$, $\ldots$, $y_n$. It is easy to see that the permutation group consisting of 
$$\varphi_{\sigma}(y_1, \ldots, y_n):=(y_{\sigma(1)}, \ldots, y_{\sigma(n)})$$
belongs to the linear isometry group of a locally symmetric polynomial Finsler metric at each point of a local neighbourhood, where $\sigma$ runs through the permutations of the indices. The permutation group is always reducible because equation
$$y_1+\ldots+y_n=0 \ \ \Leftrightarrow \ \ df=0, \ \textrm{where}\ f=u_1+\ldots+u_n$$
gives invariant hyperplanes in the tangent spaces. It is an integrable distribution with level sets of the function $f$ as integral manifolds. Since
$$\gamma_p(v,v)=\gamma_p(\varphi_{\sigma}(v), \varphi_{\sigma}(v)),$$
it also follows that a locally symmetric polynomial metric induces a locally symmetric averaged Riemannian metric (second root metric). It is investigated in the next section. 

\section{On the regularity of locally symmetric second root metrics: Riemannian manifolds}

Let $M$ be a connected Riemannian manifold with a locally symmetric second root metric $F=\sqrt{A}$. Its local characteristic polynomial must be of the form
$$
P(s_1, \ldots, s_n)=F^2(x,y)=A(x,y)=a(x) s_1^2 + b(x) s_2, 
$$
where $a(x)=c_{20}(x)$ and $b(x)=c_{01}(x)$ are smooth coefficients depending on the position. Therefore
$$
A_i(x,y)= 2 a(x) s_1 + b(x) (s_1-y_i),\ \ A_{ij}(x,y)= 2 a(x)+b(x)(1-\delta_{ij}) \ \ (i, j=1, \ldots, n).
$$
In the sense of Lemma 2 we have to check the positive definiteness of the Hessian. The positive definiteness of $A_{ij}$ obviously implies that $a(x)>0$ ($i=j=1$) and we can write the Hessian in the form
$$A_{ij}(x,y)= 2 a(x)\left(1+\frac{h(x)}{2}(1-\delta_{ij})\right),$$
where $h(x)=b(x)/a(x)$. Introducing the parameter
$$p(x)=1+\frac{h(x)}{2},$$
we are going to check the positive definiteness up to a positive conformal term:
$$g_{ij}(x)=1+\frac{h(x)}{2}(1-\delta_{ij})=\begin{bmatrix}
	1 & p(x) & p(x) & \ldots & p(x)\\
  p(x) & 1 & p(x) & \ldots & p(x)\\
	p(x) & p(x) & 1 & \ldots & p(x)\\
  \vdots & \vdots & \vdots &  & \vdots \\
  p(x) & p(x) & p(x) & \ldots & 1
	\end{bmatrix}.$$
Its characteristic polynomial is
	\[\begin{vmatrix}
	1-\lambda & p(x) & p(x) & \ldots & p(x)\\
  p(x) & 1-\lambda & p(x) & \ldots & p(x)\\
	p(x) & p(x) & 1-\lambda & \ldots & p(x)\\
  \vdots & \vdots & \vdots &  & \vdots \\
  p(x) & p(x) & p(x) & \ldots & 1-\lambda
	\end{vmatrix}.
\]
Subtracting the last row from all the rows above, we have
	\[\begin{vmatrix}
	1-\lambda-p(x) & 0 & 0 & \ldots & p(x)-1+\lambda\\
  0 & 1-\lambda-p(x) & 0 & \ldots & p(x)-1+\lambda\\
	0 & 0 & 1-\lambda-p(x) & \ldots & p(x)-1+\lambda\\
  \vdots & \vdots & \vdots &  & \vdots \\
  p(x) & p(x) & p(x) & \ldots & 1-\lambda
	\end{vmatrix}.
\]
Adding all columns to the last one gives
	\[\begin{vmatrix}
	1-\lambda-p(x) & 0 & 0 & \ldots & 0\\
  0 & 1-\lambda-p(x) & 0 & \ldots & 0\\
	0 & 0 & 1-\lambda-p(x) & \ldots & 0\\
  \vdots & \vdots & \vdots &  & \vdots \\
  p(x) & p(x) & p(x) & \ldots & 1-\lambda+(n-1)p(x)
	\end{vmatrix}
\]
and the determinant is equal to the product of the elements in its diagonal:
	\[(1-p(x)-\lambda)^{n-1}\left(1+(n-1)p(x)-\lambda\right).
\]
Therefore the matrix has eigenvalues $\lambda_1=1-p(x)$ and $\lambda_2=1+(n-1)p(x)$. It is positive definite if and only if $1-p(x)>0$  and $1+(n-1)p(x)>0$, which is equivalent to
	\begin{equation}-\,\frac{1}{n-1}<p(x)<1
\end{equation}
and we have the following result.

\begin{Prop}
The Hessian
$$A_{ij}(x,y)= 2 a(x)+b(x)(1-\delta_{ij})= 2 a(x)\left(1+\frac{h(x)}{2}(1-\delta_{ij})\right)$$
is positive definite at all nonzero elements $\displaystyle{v \in \pi^{-1}(U)\subset TM}$ if and only if
$$a(x)>0 \quad \textrm{and} \quad -\,\frac{1}{n-1}<p(x)<1,$$
where $h(x)=b(x)/a(x)$ and
$$p(x)=1+\frac{h(x)}{2}.$$
\end{Prop}

\begin{Prop} If the matrix
$$g_{ij}(x)=1+\frac{h(x)}{2}(1-\delta_{ij})$$
is positive definite then its inverse is
	\[g^{ij}(x)=\frac{1}{(1-p(x))(1+(n-1)p(x))}\,b^{ij}(x),
\]
where
	\[b^{ij}(x)=\begin{cases}
	1+(n-2)p(x) & \textrm{if }i=j,\\
	-p(x) & \textrm{if }i\neq j.
	\end{cases}
\]
\end{Prop}

\begin{Pf} If $i=j$ then
	\[\sum_{k=1}^ng_{ik}b^{kj}=\sum_{k=1}^ng_{ik}b^{ki}=g_{ii}b^{ii}+\sum_{\substack{1\leq k\leq n\\ k\neq i}}g_{ik}b^{ki}=1+(n-2)p(x)-(n-1)p(x)^2=
\]
	\[=(1-p(x))(1+(n-1)p(x)).
\]If $i\neq j$ then
	\[\sum_{k=1}^ng_{ik}b^{kj}=g_{ii}b^{ij}+g_{ij}b^{jj}+\sum_{\substack{1\leq k\leq n\\ k\neq i,\,k\neq j}}g_{ik}b^{ki}=
\]
	\[=-\,p(x)+p(x)(1+(n-2)p(x))-(n-2)p(x)^2=0.
\]
\end{Pf}
\subsection{Canonical data on Riemannian surfaces} In what follows we are interested in the special case of Riemannian surfaces with $n=2$, i.e. 
	\[g_{ij}=\begin{bmatrix}
	1 & p(x)\\[0.2cm]
	p(x) & 1
	\end{bmatrix}\quad \textrm{and} \quad g^{ij}=\frac{1}{1-p(x)^2}\,\begin{bmatrix}
	1 & -p(x)\\[0.2cm]
	-p(x) & 1
	\end{bmatrix}.
\]
Using the abbreviation $\partial_i$ ($i=1, 2$) for differentiation with respect to the position, a straightforward calculation shows that
\begin{itemize}
\item[(i)] the Christoffel symbols
\end{itemize}
	\[\Gamma_{ij}^k=\frac{1}{2}g^{kl}\left(\partial_i g_{jl}+\partial_j g_{il}-\partial_l g_{ij}\right)
\]
are $\Gamma_{12}^1=\Gamma_{12}^2=\Gamma_{21}^1=\Gamma_{21}^2=0$, 
	\[\Gamma_{11}^1=-\frac{p(x)\,\partial_1p(x)}{1-p(x)^2},\ \Gamma_{11}^2=\frac{\partial_1p(x)}{1-p(x)^2},\ \Gamma_{22}^1=\frac{\partial_2p(x)}{1-p(x)^2},\ \Gamma_{22}^2=-\frac{p(x)\,\partial_2p(x)}{1-p(x)^2},
\]
\begin{itemize}
\item[(ii)] the components of the Riemannian curvature tensor 
\end{itemize}
\[R_{ijk}^l=\partial_i \Gamma_{jk}^l-\partial_j \Gamma_{ik}^l+\Gamma_{ir}^l\Gamma_{jk}^r-\Gamma_{jr}^l\Gamma_{ik}^r
\]
are $R_{111}^1=R_{111}^2=R_{112}^1=R_{112}^2=0$,
\begin{eqnarray*}
&& R_{121}^1=\frac{p(x)\,\big(1-p(x)^2\big)\,\partial_{1,2}p(x)+p(x)^2\,\partial_{1}p(x)\,\partial_{2}p(x)}{(1-p(x)^2)^2},\\
&& R_{121}^2=-\frac{(1-p(x)^2)\,\partial_{1,2}p(x)+p(x)\,\partial_{1}p(x)\,\partial_{2}p(x)}{(1-p(x)^2)^2},\\
&& R_{122}^1=\frac{(1-p(x)^2)\,\partial_{1,2}p(x)+p(x)\,\partial_{1}p(x)\,\partial_{2}p(x)}{(1-p(x)^2)^2},\\
&& R_{122}^2=-\frac{p(x)\,\big(1-p(x)^2\big)\,\partial_{1,2}p(x)+p(x)^2\,\partial_{1}p(x)\,\partial_{2}p(x)}{(1-p(x)^2)^2}, \\
&& R_{211}^1=-\frac{p(x)\,\big(1-p(x)^2\big)\,\partial_{1,2}p(x)+p(x)^2\,\partial_{1}p(x)\,\partial_{2}p(x)}{(1-p(x)^2)^2},\\
&& R_{211}^2=\frac{(1-p(x)^2)\,\partial_{1,2}p(x)+p(x)\,\partial_{1}p(x)\,\partial_{2}p(x)}{(1-p(x)^2)^2},\\
&& R_{212}^1=-\frac{(1-p(x)^2)\,\partial_{1,2}p(x)+p(x)\,\partial_{1}p(x)\,\partial_{2}p(x)}{(1-p(x)^2)^2}, \\
&& R_{212}^2=\frac{p(x)\,(1-p(x)^2)\,\partial_{1,2}p(x)+p(x)^2\,\partial_{1}p(x)\,\partial_{2}p(x)}{(1-p(x)^2)^2}
\end{eqnarray*}
and $R_{221}^1=R_{221}^2=R_{222}^1=R_{222}^2=0$. The lowered components 
	\[R_{ijkl}=g_{lm}\,R_{ijk}^m
\]
are $R_{1111}=R_{1112}=R_{1121}=R_{1122}=R_{1211}=R_{1222}=0$,
\begin{eqnarray*}
&& R_{1212}=-\frac{(1-p(x)^2)\,\partial_{1,2}p(x)+p(x)\,\partial_{1}p(x)\,\partial_{2}p(x)}{1-p(x)^2},\\
&& R_{1221}=\frac{(1-p(x)^2)\,\partial_{1,2}p(x)+p(x)\,\partial_{1}p(x)\,\partial_{2}p(x)}{1-p(x)^2},\\
&&R_{2112}=\frac{(1-p(x)^2)\,\partial_{1,2}p(x)+p(x)\,\partial_{1}p(x)\,\partial_{2}p(x)}{1-p(x)^2},\\
&&R_{2121}=-\frac{(1-p(x)^2)\,\partial_{1,2}p(x)+p(x)\,\partial_{1}p(x)\,\partial_{2}p(x)}{1-p(x)^2}
\end{eqnarray*}
and $R_{2111}=R_{2122}=R_{2211}=R_{2212}=R_{2221}=R_{2222}=0$,
\begin{itemize}
\item[(iii)] the components of the Ricci tensor
\end{itemize}
	\[{\textrm{Ric}}_{ij}=g^{km}\,R_{ijkm}
\]
are
	\begin{align*}
	{\textrm{Ric}}_{11}&={\textrm{Ric}}_{22}=\frac{(1-p(x)^2)\,\partial_{1,2}p(x)+p(x)\,\partial_{1}p(x)\,\partial_{2}p(x)}{(1-p(x)^2)^2},\\
{\textrm{Ric}}_{12}&={\textrm{Ric}}_{21}=\frac{p(x)\,(1-p(x)^2)\,\partial_{1,2}p(x)+p(x)^2\,\partial_{1}p(x)\,\partial_{2}p(x)}{(1-p(x)^2)^2},
\end{align*}
\begin{itemize}
\item[(iv)] the scalar curvature 
\end{itemize}
	\[S=g^{ij}{\textrm{Ric}}_{ij}
\]
is
	\[S=\frac{2(1-p(x)^2)\,\partial_{1,2}p(x)+2p(x)\,\partial_{1}p(x)\,\partial_{2}p(x)}{(1-p(x)^2)^2},
\]
\begin{itemize}
\item[(v)] the Gaussian curvature 
\end{itemize}
	\[K=\frac{1}{2}S\]
	is
\begin{equation}
\label{Gcurv2D}
K=\frac{\big(1-p(x)^2\big)\,\partial_{1,2}p(x)+p(x)\,\partial_{1}p(x)\,\partial_{2}p(x)}{(1-p(x)^2)^2}.
\end{equation}

The Gaussian curvature completely determines the curvature tensor of a Riemannian 2-manifold. In what follows we are going to find a local solution function $p(x)$ of equation \eqref{Gcurv2D} for surfaces with constant curvature.

\subsection{Riemannian surfaces with constant curvature} Given a real constant $k\in\rr$, let's solve equation
\begin{equation}
\label{curveq_k}
\frac{\big(1-p(x_1,x_2)^2\big)\,\partial_{1,2}p(x_1,x_2)+p(x_1,x_2)\,\partial_{1}p(x_1,x_2)\,\partial_{2}p(x_1,x_2)}{(1-p(x_1,x_2)^2)^2}=k,
\end{equation}
where $-1<p(x_1,x_2)<1$. 
\begin{itemize}
\item[(i)] If $k=0$, then we are looking for the possible solutions in the special form $p(x_1,x_2)=f_1(x_1)\,f_2(x_2)$.
\end{itemize}
Since 
$$\partial_{1}p(x_1,x_2)= f_1'(x_1)\,f_2(x_2), \ \partial_{2}p(x_1,x_2)= f_1(x_1)\,f_2'(x_2), \ \partial_{1,2}p(x_1,x_2)=f_1'(x_1)\,f_2'(x_2),$$
it follows that equation \eqref{curveq_k} reduces to 
\[f_1'(x_1)\,f_2'(x_2)=0
\]
and $p(x_1,x_2)=c_1\,f_2(x_2)$ or $p(x_1,x_2)=c_2\,f_1(x_1)$ are possible solutions.
\begin{itemize}
\item[(ii)] If $k\neq 0$, then both
\end{itemize}
\[p(x_1,x_2)=2\,\tanh^2\left(c_1\,x_1-\frac{kx_2}{c_1}+c_2\right)-1
\]
and
	\[p(x_1,x_2)=1-2\,\tanh^2\left(c_1\,x_1+\frac{kx_2}{c_1}+c_2\right)
\]
are possible solutions, where $c_1,c_2\in\rr$, $c_1\neq 0$.

\section{On the regularity of locally symmetric fourth root metrics: Finsler surfaces}

In what follows we formulate necessary and sufficient conditions for a locally symmetric fourth root metric in 2D to be positive definite. They are given in terms of the coefficients of the polynomial metric to make checking the positive definiteness as simple and direct as possible. Explicit examples will also be presented. 

\subsection{The Hessian} Let $M$ be a two-dimensional connected Finsler manifold (Finsler surface) with a locally symmetric fourth root metric $F=\sqrt[4]{A}$. Its local characteristic polynomial must be
\begin{equation}
P(s_1, s_2)=F^4(x,y)=A(x,y)=a(x) (y_1+y_2)^4 + b(x) (y_1+y_2)^2 y_1 y_2 + c(x) (y_1y_2)^2 \label{poly1}
\end{equation}
where $a(x)=c_{40}(x)$, $b(x)=c_{21}(x)$ and $c(x)=c_{02}(x)$. Writing $A$ in the usual form
\begin{equation} \label{poly2}
A(x,y)=l(x) \left(y_1^4+y_2^4\right) + m(x) \left(y_1^3y_2+y_1y_2^3 \right) + n(x) y_1^2y_2^2, 
\end{equation}
we can check that the coefficients are 
\begin{equation} l(x)= a(x), \ m(x)= 4 a(x) + b(x), \ n(x)= 6 a(x) + 2 b(x) + c(x),
\end{equation}
i.e. they are obtained by the regular linear transformation
\[
\left[
\begin{matrix}
l \\ m \\ n
\end{matrix}
\right]
=
\left[
\begin{matrix}
1 & 0 & 0 \\ 4 & 1 & 0 \\ 6 & 2 & 1
\end{matrix}
\right]
\left[
\begin{matrix}
a \\ b \\ c
\end{matrix}
\right].
\]
Therefore, the coefficients $a(x), b(x)$, $c(x)$ are uniquely determined by $l(x)$, $m(x)$, $n(x)$ and vice versa. Differentiating \eqref{poly2},
\begin{align*}
A_1(x,y)&= 4 l(x) y_1^3 + 3 m (x) y_1^2 y_2+ 2n (x)y_1 y_2^2 + m (x) y_2^3, \\
A_2(x,y)&= m (x)y_1^3 + 2n (x) y_1^2 y_2+ 3m (x)y_1 y_2^2 + 4l (x) y_2^3, \\
\intertext{and the second-order derivatives are}
A_{11}(x,y)&= 12 l(x) y_1^2 + 6m(x) y_1 y_2 + 2n(x) y_2^2,\\
A_{12}(x,y)&=A_{21}(x,y)= 3m(x)y_1^2 + 4n(x)  y_1 y_2 + 3m(x)y_2^2,  \\
A_{22}(x,y)&= 2n(x)y_1^2 + 6m(x) y_1 y_2 + 12 l(x) y_2^2 .
\end{align*}
The Hessian 
$$A_{ij}=\left[
\begin{matrix}
A_{11} & A_{12}\\ A_{21} & A_{22}
\end{matrix}
\right] $$
is of the form
$$A_{11}= \left[
\begin{matrix}
y_1 & y_2
\end{matrix}
\right]
\left[
\begin{matrix}
12l & 3m\\ 3m & 2n
\end{matrix}
\right]
\left[
\begin{matrix}
y_1 \\ y_2
\end{matrix}
\right], \ A_{12}= A_{21}=\left[
\begin{matrix}
y_1 & y_2
\end{matrix}
\right]
\left[
\begin{matrix}
3m & 2n\\ 2n & 3m
\end{matrix}
\right]
\left[
\begin{matrix}
y_1 \\ y_2
\end{matrix}
\right], $$
$$A_{22}= \left[
\begin{matrix}
y_1 & y_2
\end{matrix}
\right]
\left[
\begin{matrix}
2n & 3m\\ 3m & 12l
\end{matrix}
\right]
\left[
\begin{matrix}
y_1 \\ y_2
\end{matrix}
\right].
$$

\begin{Lem} \label{lem1} $A_{ij}$ is positive definite if and only if 
\begin{itemize}
\item $A_{11}>0$ and
\item $\det A_{ij}=A_{11}A_{22}-A_{12}^2>0$, where
\end{itemize}
\begin{gather}
\det A_{ij}(y_1,y_2)= (24ln-9m^2) \left(y_1^4+y_2^4\right)+(72lm-12mn)\left(y_1^3y_2+y_1y_2^3 \right) \label{det} \\ 
+(144l^2+18m^2-12n^2)y_1^2y_2^2. \notag
\end{gather}
\end{Lem}

\subsection{Necessary conditions of positive definiteness}\label{nec} Suppose that the matrix $A_{ij}$ is positive definite and let us consider the following necessary conditions. 

\begin{itemize}
\item Condition $A_{11}>0$ is obviously equivalent to
\begin{equation} \label{posi01}
\left[
\begin{matrix}
y_1 & y_2
\end{matrix}
\right]
\left[
\begin{matrix}
12l & 3m\\ 3m & 2n
\end{matrix}
\right]
\left[
\begin{matrix}
y_1 \\ y_2
\end{matrix}
\right] > 0 
 \Longleftrightarrow 
\left\{\begin{alignedat}{1}
12l &> 0 \\
24 ln - 9 m^2 &>0
\end{alignedat}\right.
 \Longleftrightarrow 
\left\{\begin{alignedat}{1}
\mathbf{l}\ &\mathbf{> 0} \\
\mathbf{8 ln}\  &\mathbf{>3m^2}
\end{alignedat}\right..
\end{equation}
Especially, $\mathbf{n>0}$ and condition \eqref{posi01} automatically provides the positive definiteness of $A_{ij}$ along the coordinate axes because of
$$\det A_{ij}(t, 0)=\det A_{ij}(0,t)=(24ln-9m^2)t^4.$$

\item Using the diagonal and the antidiagonal directions, we must have
$$\det A_{ij}(t,\pm t)=12(6l-n)(2l \pm 2m+n)t^4>0$$
and a simple addition implies that $12(6l-n)(4l+2n)t^4>0$. Since $l$ and $n$ are strictly positive by \eqref{posi01}, $6l-n >0$, i.e. $\mathbf{6l>n}$ must hold. Therefore we have 
\begin{equation} \label{OKVcond}
\mathbf{l > 0} \quad \textrm{and} \quad \mathbf{\frac{3m^2}{8l} < n < 6l}
\end{equation}
as bounds for the coefficient $n(x)$. Further necessary conditions can be given by using the positivity along the diagonal and the antidiagonal directions in a separated way:  
\begin{equation}
\label{pluscond}
\left\{\begin{alignedat}{1}
2l+2m+n &> 0 \\
2l-2m+n &> 0
\end{alignedat}\right.
\Longleftrightarrow
\left\{\begin{alignedat}{1}
2l+n &> -2m \\
2l+n &> 2m
\end{alignedat}\right.
\Longleftrightarrow
\mathbf{2\left|m\right| < 2l+n}.
\end{equation}
Combining the previous conditions we also have that
\begin{equation}\label{concond}
2\left|m\right| < 2l+n < 2l+6l=8l \ \Longrightarrow \mathbf{\left|m\right|<4l}.                                               
\end{equation} 
\end{itemize}

In the paper \cite{OKV}, the authors presented condition \eqref{OKVcond} to support theoretical results for such a metric to be an easy-to-compute example for a generalized Berwald metric. The results related to the special metric obviously hold, but condition \eqref{OKVcond} is not equivalent to the positive definiteness of the metric. Indeed, the coefficients $l=4$, $m=6$ and $n=4$ show that condition \eqref{OKVcond} can be satisfied without satisfying the necessary condition \eqref{pluscond} for the positive definiteness of the matrix $A_{ij}$. Using a continuity argument, we can also construct indefinite polynomial metrics with non-constant coefficients: $l(x)\approx 4$,  $m(x)\approx 6$ and $n(x)\approx 4$.

\subsection{Some examples} In most cases of the following examples we use constant coefficients but a continuity argument  allows us to extend the definiteness property for a polynomial metric with non-constant coefficients. 

\begin{Exm}{\emph{The coefficients $l=1$, $m=\sqrt{7}$ and $n=3$ show that condition \eqref{OKVcond} can be satisfied without satisfying the necessary condition \eqref{pluscond} for the positive definiteness of the matrix $A_{ij}$.}}
\end{Exm}

\begin{Exm} \label{exe1} {\emph{Setting $l=4$, $m=6$ and $n=5$, it can easily be seen that all the necessary conditions \eqref{OKVcond} and \eqref{pluscond} are satisfied. By substituting into \eqref{det}, the determinant of the Hessian is
\[ \det A_{ij}(y_1,y_2)=
156(y_1^4+y_2^4)+1368(y_1^3y_2+y_1y_2^3)+2652y_1^2y_2^2 \]
and  we do not always get a positive value: $\det A_{ij}(1,-2)=-420$.}}
\end{Exm}

\begin{Exm} \label{exe2} {\emph{Let us choose $l=1$, $m=2$ and $n=3$. Although all the necessary conditions \eqref{OKVcond} and \eqref{pluscond} are satisfied, the previous example shows that they do not form a sufficient system of conditions for the positive definiteness in general. Therefore we have to check inequality $\det A_{ij}>0$ in a direct way ($A_{11}>0$ is given by $l > 0$ and $8ln > 3m^2$):
$$\det A_{ij} =36(y_1^4+y_2^4)+72(y_1^3y_2+y_1y_2^3)+108y_1^2y_2^2=36 \left( y_1^2+y_1y_2+y_2^2 \right)^2,
$$
where the quadratic form $y_1^2+y_1y_2+y_2^2$ is obviously positive definite. So is the matrix $A_{ij}$.}}
\end{Exm}

\begin{Exm} \label{exe3} {\emph{Let us set $l=1$, $m=1$ and $n=2$. Since the necessary conditions \eqref{OKVcond} and \eqref{pluscond}  are satisfied, the positive definiteness depends on the positivity of $\det A_{ij}$ along directions such that $y_1y_2\neq 0$. Plugging into \eqref{det} yields
\[ \det A_{ij}(y_1,y_2)=39(y_1^4+y_2^4)+48(y_1^3y_2+y_1y_2^3)+114y_1^2y_2^2, \]
about which neither positivity nor reducibility is obvious to determine. Dividing by $3y_1y_2$ and introducing the new variable $z:=\dfrac{y_1}{y_2}+\dfrac{y_2}{y_1}$, we have that 
$$13 \left[\left(\dfrac{y_1}{y_2}\right)^2+\left(\dfrac{y_2}{y_1} \right)^2\right]+16\left[\dfrac{y_1}{y_2}+\dfrac{y_2}{y_1} \right] +38
=13(z^2-2)+16z+38 = 13z^2+16z+12.$$
It is easy to see that the polynomial is positive everywhere because it has a positive main coefficient and negative discriminant $-368$. So is $ \det A_{ij}$, making $A_{ij}$ positive definite.}}
\end{Exm}

\begin{Exm}\label{exe7}
{\emph{(Positive definite generalized Berwald surfaces) According to the characterization of generalized Berwald surfaces with locally symmetric fourth root metrics in \cite{OKV}, the choice $n(x)\approx 1$ (i.e. $n(x)$ is close enough to one), $l(x)=1/4$ and
$$m(x)=\sqrt{n(x)-1/2+\frac{1}{36}(2n(x)-3)^2}$$
gives a generalized Berwald surface. Using a continuity argument it is enough to evaluate the principal minors of the matrix $A_{ij}$ under the choice $l=1/4$, $n=1$ and 
$$m=\sqrt{1-1/2+\frac{1}{36}(2\cdot 1-3)^2}=\sqrt{19}/6$$
to see the positive definiteness:}} 
$$A_{11}= \left[
\begin{matrix}
y^1 & y^2
\end{matrix}
\right]
\left[
\begin{matrix}
3 & \sqrt{19}/2\\ \sqrt{19}/2 & 2
\end{matrix}
\right]
\left[
\begin{matrix}
y^1 \\ y^2
\end{matrix}
\right],$$ 
{\emph{where}}  
$$\left[
\begin{matrix}
3 & \sqrt{19}/2\\ \sqrt{19}/2 & 2
\end{matrix}
\right]$$
{\emph{is a positive definite matrix, 
$$\det A_{ij}(t,0)=\det A_{ij}(0, t)=\frac{5}{4}t^4$$
and
$$\det A_{ij}(y^1,y^2)=(y^2)^4\left(\frac{5}{4}\left(\frac{y^1}{y^2}\right)^4+\sqrt{19}\left(\frac{y^1}{y^2}\right)^3+\frac{13}{2}\left(\frac{y^1}{y^2}\right)^2+\sqrt{19}\left(\frac{y^1}{y^2}\right)+\frac{5}{4}\right)>0$$
because 
$$\frac{5}{4}t^4+\sqrt{19}t^3+\frac{13}{2}t^2+\sqrt{19}t+\frac{5}{4}$$
is an irreducible palindromic polynomial (see subsection \ref{palind}).}}
\end{Exm}

In what follows we are going to apply the method of palindromic polynomials (see Examples \ref{exe3} and \ref{exe7}) in general. 

\subsection{The definiteness polynomial}\label{palind}

As \eqref{posi01} implies the positive definiteness of $A_{ij}$ along the coordinate axes, it is enough to investigate the case of $y_1y_2\neq 0$. After dividing \eqref{det} by $3y_1^2y_2^2$, $\det A_{ij}>0$ is equivalent  to
\[ (8ln-3m^2) \left[\left(\dfrac{y_1}{y_2}\right)^2+\left(\dfrac{y_2}{y_1} \right)^2\right]+(24lm-4mn)\left[\dfrac{y_1}{y_2}+\dfrac{y_2}{y_1} \right] +48l^2+6m^2-4n^2 > 0. \]
Introducing the reciprocal sum $z:=\dfrac{y_1}{y_2}+\dfrac{y_2}{y_1}$ as a new variable, it follows that $\left|z\right|\geq 2$ and we can introduce the (local) \textbf{definiteness polynomial} 
\begin{align}
P(z)&= (8ln-3m^2) \left(z^2-2\right)+(24lm-4mn)z+48l^2+6m^2-4n^2, \notag \\
P(z)&= (8ln-3m^2)z^2+(24lm-4mn)z+48l^2+12m^2-4n^2-16ln \label{quadraticpol}
\end{align} 
of the metric. In terms of the definiteness polynomial, Lemma \ref{lem1} can be reformulated as follows.
 
\begin{Lem} \label{lem-posdefpol} $A_{ij}$ is positive definite if and only if
\begin{itemize}
\item \eqref{posi01} holds and
\item $P(z)>0$ for all values where $\left|z\right|\geq 2$, i.e. $P$ is irreducible over the reals \emph{(irreducible case)}
 or $P$ is reducible having roots in the open interval $\left]-2,2\right[$ \emph{(reducible case)}.
\end{itemize}
\end{Lem}

In what follows, we suppose that necessary conditions \eqref{OKVcond} and \eqref{pluscond} hold. So does \eqref{concond}. Especially, $0<l$ and $\left|m\right|<4l$ (we will see that these are as tight as possible, and the additional conditions are given for the coefficient $n$). At first let us investigate the discriminant
\begin{align}
\Delta:&= (24lm-4mn)^2-4\left(8ln-3m^2\right)\left(48l^2+12m^2-4n^2-16ln\right) \label{disc}\\
 &= 16\left(9m^2-12ln-2n^2\right)\left(8l^2-4ln+m^2\right) =: 16\, \Delta_1 \, \Delta_2 \notag
\end{align}
of the definiteness polynomial.

\begin{Lem}[The factors of the discriminant] \label{lem-disc} Necessary conditions \eqref{OKVcond} and \eqref{pluscond} imply that  
\begin{equation} \label{Delta1-2-magn}
\Delta_1=9m^2-12ln-2n^2 \ < \ 8l^2-4ln+m^2 = \Delta_2.
\end{equation} 
Furthermore, sign conditions for $\Delta_1$ and $\Delta_2$ can be formulated in terms of lower or upper bounds for the coefficient $n$ in the following way:
\begin{align}
\Delta_1=9m^2-12ln-2n^2 \lessgtr 0 & \Longleftrightarrow  
  \dfrac{3}{2}\sqrt{4l^2+2m^2}-3l \lessgtr n \label{Delta1-ineq}\\[8pt] 
\Delta_2=8l^2-4ln+m^2 \lessgtr 0 & \Longleftrightarrow  
   \dfrac{8l^2+m^2}{4l} \lessgtr n. \label{Delta2-ineq}
\end{align}
\end{Lem}

\begin{proof} \eqref{Delta1-2-magn} is equivalent to the necessary condition $2\left|m\right|<2l+n$ through
\begin{align*}
9m^2-12ln-2n^2  &< 8l^2-4ln+m^2  \\
-8l^2-8ln-2n^2 &< -8m^2 \\
4l^2+4ln+n^2 & > 4m^2 \\
2l+n & > 2\left|m\right|.
\end{align*}
\eqref{Delta1-ineq} and \eqref{Delta2-ineq} are obtained by equivalent rearrangements such as
$$\Delta_1=9m^2-12ln-2n^2=9m^2-2(3l+n)^2+18l^2 \lessgtr 0,$$
where $3l+n>0$. The second one \eqref{Delta2-ineq} is straightforward because of $l>0$. 
\end{proof}

We have already encountered quite a few bounds for the coefficient $n$; let's see how they compare to each other.

\begin{Lem}[The magnitude of bounds] \label{lem-magnbound} Necessary conditions \eqref{OKVcond} and \eqref{pluscond} imply that 
\begin{equation} \label{boundsforn}
0 \ \leq \ \dfrac{3m^2}{8l} \ \leq \ \dfrac{3}{2}\sqrt{4l^2+2m^2}-3l \ < \ \dfrac{8l^2+m^2}{4l} \ < \ 6l.
\end{equation}
\end{Lem}

\begin{proof} All the inequalities are consequences of the necessary condition $\left|m\right|<4l$ as follows. 
\begin{itemize} 
\item[(a)] $\hfill \begin{aligned}[t]
\dfrac{3m^2}{8l}  &\leq \dfrac{3}{2}\sqrt{4l^2+2m^2}-3l \hspace{1cm} \bigg/\cdot \dfrac{8}{3} l>0 \\
m^2  &\leq 4l\sqrt{4l^2+2m^2}-8l^2  \\
8l^2+m^2  &\leq 4l\sqrt{4l^2+2m^2} \\
64l^2+16l^2m^2+m^4 &\leq 64l^4+32l^2m^2 \\
0 &\leq 16l^2m^2-m^4 = m^2(4l+m)(4l-m).
\end{aligned} \hfill$ \vspace{8pt}

\item[(b)]  $\hfill \begin{aligned}[t]
\dfrac{3}{2}\sqrt{4l^2+2m^2}-3l &< \dfrac{8l^2+m^2}{4l} \hspace{1cm} \bigg/\cdot 4l>0 \\
6l\sqrt{4l^2+2m^2}-12l^2 &< 8l^2+m^2 \\
6l\sqrt{4l^2+2m^2} &< 20l^2+m^2 \\
144l^4+72l^2m^2 &< 400l^4+40l^2m^2+m^4 \\
0 &< 256l^4-32l^2m^2+m^4 = (4l+m)^2(4l-m)^2.
\end{aligned} \hfill$ \vspace{8pt}

\item[(c)] $\hfill \begin{aligned}[t]
\dfrac{8l^2+m^2}{4l} &< 6l \hspace{1cm} \bigg/\cdot 4l>0 \\
8l^2+m^2 &< 24l^2 \\
m^2 &< 16l^2 \\
\left|m\right| &< 4l. 
\end{aligned} \hfill$ 
\end{itemize}
\end{proof}

\subsection{The irreducible case} Irreducibility of the definiteness polynomial is equivalent to the discriminant being $\Delta=\Delta_1\, \Delta_2<0$. By Lemma \ref{lem-disc}, since $\Delta_1<\Delta_2$, this is only possible in case of
\[
\left.
\begin{aligned}
 \Delta_1=9m^2-12ln-2n^2 &< 0 \\[16pt]
 \Delta_2=8l^2-4ln+m^2 &> 0
\end{aligned}
\right\}
\ \Longleftrightarrow \
\left\{
\begin{aligned}
 \dfrac{3}{2}\sqrt{4l^2+2m^2}-3l &< n \\
 \dfrac{8l^2+m^2}{4l} &> n.
\end{aligned}
\right.
\]
Lemma \ref{lem-magnbound} shows that it is indeed a possible case (corresponding to case (b) in the inequality chain). So, we have the following result.

\begin{Lem}[The irreducible case] \label{lem-irredcase} Necessary conditions \eqref{OKVcond} and \eqref{pluscond} together with
\begin{equation} \label{irred}
\dfrac{3}{2}\sqrt{4l^2+2m^2}-3l < n < \dfrac{8l^2+m^2}{4l}
\end{equation} 
imply that the definiteness polynomial is irreducible over the reals, i.e. $A_{ij}$ is positive definite. 
\end{Lem}

\subsection{The reducible case} If $P$ is reducible, it must have its roots inside the open interval $\left]-2,2\right[$ for $A_{ij}$ to be positive definite. As \eqref{posi01} implies that $P$ has a positive main coefficient (and thus a monotonically increasing derivative), the condition on the roots is equivalent to
\begin{itemize}
\item[(a)] $P(\pm 2)>0$ and
\item[(b)] the slopes of $P$ flipping sign between $\pm 2$, i.e. $P'(-2)<0<P'(2)$.
\end{itemize}
Since $P(\pm 2)=\det A_{ij}(1, \pm 1)=12(6l-n)(2l\pm 2m+n)$, (a) is implied by the necessary conditions \eqref{OKVcond} and \eqref{pluscond}. Using condition $\left|m\right|<4l$ again,
\begin{align*}
P'(-2)=2(8ln-3m^2)(-2)+24lm-4mn &< 0 \\
-8ln+3m^2+6lm-mn &< 0 \\
 3m(m+2l) &< n(8l+m) & /\div (8l+m)>0 \\
 \dfrac{3m(m+2l)}{8l+m} &< n \\[6pt]
 P'(2)=2(8ln-3m^2)2+24lm-4mn &> 0 \\
8ln-3m^2+6lm-mn &> 0 \\
n(8l-m) &> 3m(m-2l) & /\div (8l-m)>0 \\
 n &> \dfrac{3m(m-2l)}{8l-m} 
\end{align*}
So, the slope conditions are equivalent to lower bounds for the coefficient $n$. About their magnitude, we can say the following. 

\begin{Lem}[The magnitude of bounds II] \label{label-bounds2} Necessary conditions \eqref{OKVcond} and \eqref{pluscond} imply that the slope conditions are  equivalent to the following lower bounds for the coefficient $n$:
\begin{equation}
\left.\begin{aligned}
\textrm{if } m \leq 0: \ &  \dfrac{3m(m-2l)}{8l-m} \\
\textrm{if } m \geq 0: \ &  \dfrac{3m(m+2l)}{8l+m} \\
\end{aligned}\right\} <n.
\end{equation}
Denoting the value of this bound by $\bigstar$, it fits in the inequality chain of Lemma \ref{lem-magnbound} as
\begin{equation} \label{boundsforn2}
0 \ \leq \ \dfrac{3m^2}{8l} \ \leq \ \dfrac{3}{2}\sqrt{4l^2+2m^2}-3l \ \leq \ \bigstar \ < \dfrac{8l^2+m^2}{4l} \ < \ 6l.
\end{equation} 
\end{Lem}

\begin{proof} Comparing the bounds leads to
\begin{align*}
\dfrac{3m(m-2l)}{8l-m} &\lessgtr \dfrac{3m(m+2l)}{8l+m} & \big/\cdot (8l-m)(8l+m)>0 \\
3m(m-2l)(8l+m) &\lessgtr 3m(m+2l)(8l-m) \\
-m(16l^2-m^2) &\lessgtr m(16l^2-m^2) & /\div (16l^2-m^2)>0 \\
-m &\lessgtr m,
\end{align*}
giving the first formula of the Lemma. For the second one, we need to prove two inequalities. Both of them follow from the necessary condition $\left|m\right|<4l$. We only consider the case $m\geq 0$ (the case $m\leq 0$ is analogous).

\begin{enumerate}
\item[(d)] $\hfill \begin{aligned}[t]
\dfrac{3}{2}\sqrt{4l^2+2m^2}-3l &\leq \dfrac{3m(m+2l)}{8l+m} =\bigstar \hspace{1cm} \bigg/\cdot \dfrac{2}{3}(8l+m)>0 \\
\sqrt{4l^2+2m^2}(8l+m) &\leq 2m(m+2l)+2l(8l+m) \\
\sqrt{4l^2+2m^2}(8l+m) &\leq 16l^2+6lm+2m^2\\
256l^4+64l^3m+132l^2m^2+32lm^3+2m^4 &\leq 256l^4+192l^3m+100l^2m^2+24lm^3+4m^4 \\
0 &\leq 128l^3m-32l^2m^2-8lm^3+2m^4 \\
0 &\leq 2m(4l+m)(4l-m)^2.
\end{aligned} \hfill$ \vspace{8pt}
\item[(e)] 
$\hfill \begin{aligned}[t]
\bigstar=\dfrac{3m(m+2l)}{8l+m} &< \dfrac{8l^2+m^2}{4l} \hspace{1cm} /\cdot 4l(8l+m)>0 \\
3m(m+2l)4l&<(8l^2+m^2)(8l+m) \\
24l^2m+12lm^2 &< 64l^3+8l^2m+8lm^2+m^3 \\
0 &< 64l^3-16l^2m-4lm^2+m^3 \\
0 &< (4l+m)(4l-m)^2
\end{aligned} \hfill$ 
\end{enumerate} 
\end{proof}

Now let's discuss the conditions of the reducibility of $P$, i.e. the discriminant being $\Delta=\Delta_1 \, \Delta_2 \geq 0$. Using Lemma \ref{lem-disc}, this can happen if $0 \leq \Delta_1 < \Delta_2$ or $\Delta_1 < \Delta_2 \leq 0$.
\begin{itemize}
\item In the first case, by Lemma \ref{lem-disc} we have
\begin{align*}
\Delta_1=9m^2-12ln-2n^2 \geq 0 & \Longleftrightarrow  
  \dfrac{3}{2}\sqrt{4l^2+2m^2}-3l \geq n \\[8pt] 
\Delta_2=8l^2-4ln+m^2 > 0 & \Longleftrightarrow  
   \dfrac{8l^2+m^2}{4l} > n.
\end{align*}
Combining the first one (which is sharper) with $\bigstar$, we get
\[ \bigstar < n \leq \dfrac{3}{2}\sqrt{4l^2+2m^2}-3l,\]
which contradicts the inequality chain from \eqref{boundsforn2}. Therefore this case can never happen.

\item In the second case, by Lemma \ref{lem-disc} we have
\begin{align*}
\Delta_1=9m^2-12ln-2n^2 < 0 & \Longleftrightarrow  
  \dfrac{3}{2}\sqrt{4l^2+2m^2}-3l < n \\[8pt] 
\Delta_2=8l^2-4ln+m^2 \leq 0 & \Longleftrightarrow  
   \dfrac{8l^2+m^2}{4l} \leq n.
\end{align*}
\end{itemize}
Combining them with $\bigstar$ (and the upper bound $n<6l$), the bounds are
\[ \dfrac{8l^2+m^2}{4l} \leq n < 6l. \]
Thus we have proved the following.

\begin{Lem}[The reducible case] \label{lem-redcase} Necessary conditions \eqref{OKVcond} and \eqref{pluscond} together with
\begin{equation} \label{red}
 \dfrac{8l^2+m^2}{4l} \leq n < 6l
 \end{equation}
imply that the definiteness polynomial is reducible over the reals with roots inside the  open interval $\left]-2, 2\right[$, i.e. $A_{ij}$ is positive definite. 
\end{Lem}

\subsection{Equivalent characterization of positive definiteness} Combining (and strengthening) the results of Lemmas \ref{lem-posdefpol}, \ref{lem-irredcase} and \ref{lem-redcase}, the main theorem of this section is the following.

\begin{Thm}[Positive definiteness of locally symmetric fourth root metrics] \label{2dmaintheorem} The fourth root $F$ of a smooth, pointwise polynomial function
\[ A(x,y)=l(x) \left[y_1^4+y_2^4\right] + m(x) \left[y_1^3y_2+y_1y_2^3 \right] + n(x) y_1^2y_2^2 \]
is a \emph{(}positive definite\emph{)} Finsler metric if and only if 
\begin{equation} \label{nbounds}
\boxed{\dfrac{3}{2}\sqrt{4l^2+2m^2}-3l < n < 6l.}
\end{equation}
\end{Thm}

\begin{proof} By Lemma \ref{lemma:1} and Lemma \ref{lemma:2},  $F$ is a (positive definite) Finsler metric if and only if the Hessian $A_{ij}$ is positive definite at all nonzero elements $\displaystyle{v \in \pi^{-1}(U)\subseteq TM}$. By Lemma \ref{lem-posdefpol}, Lemma \ref{lem-irredcase} and \ref{lem-redcase}, this is equivalent to the necessary conditions \eqref{OKVcond} and \eqref{pluscond}  together with \eqref{nbounds}. We are going to show that \eqref{nbounds} implies all of the necessary conditions.

$n<6l$ is contained explicitly in \eqref{nbounds}. Adding $3l$ to the inequality chain 
\[ 0 \leq \dfrac{3}{2}\sqrt{4l^2+2m^2} < n+3l < 9l, \]
implying $l>0$ and $n+3l>0$. Therefore the lower inequality of \eqref{nbounds} is equivalent to $\Delta_1<0$ as we have seen in the proof of \eqref{Delta1-ineq} (see Lemma \ref{lem-disc}) under $l>0$ and $n+3l>0$. Using $n<6l$, we get
\begin{gather*}
\Delta_1=9m^2-12ln-2n^2 < 0 \\
9m^2 < 12ln+2n^2 =2n(6l+n) < 2n(6l+6l) = 24ln \\
3m^2 < 8ln.
\end{gather*}
Rearranging $\Delta_1<0$ it follows that 
\begin{equation}\label{mabslow}
m^2 < \dfrac{1}{9} \left( 12ln+2n^2 \right).
\end{equation}
From the upper part of \eqref{nbounds},
\begin{align}
n &< 6l \notag \\
0 &< (6l-n)^2=36l^2-12ln+n^2 \notag\\
48ln+8n^2 &<36l^2+36ln+9n^2 \notag\\
4(12ln+2n^2) &< 9(2l+n)^2 \notag\\
\dfrac{1}{9}\left(12ln+2n^2 \right) &< \dfrac{1}{4}(2l+n)^2 \label{mabsup}
\end{align}
Combining \eqref{mabslow} and \eqref{mabsup} yields
\[ m^2 < \dfrac{1}{4}(2l+n)^2 \ \Longleftrightarrow \ 2\left|m\right| < 2l+n \]
which completes the list of necessary conditions we have used throughout our investigations.
\end{proof}

\subsection{Constructing examples} From our results, it is straightforward how to construct a locally symmetric fourth root metric (locally): choose the coefficients $l,m,n$ of $A=F^4$ in \eqref{poly2} such that they satisfy
\[ \dfrac{3}{2}\sqrt{4l^2+2m^2}-3l < n < 6l. \]
The following table contains the possible intervals for the values of $n$ after choosing $l$ and $m$ as some small integers:
\begin{center} \renewcommand\arraystretch{1.2}
\begin{tabular}{|c||c|c|c|c|}
\hline 
\diagbox{$|m|$}{$l$} & 1 & 2 & 3 & 4 \\ 
\hline  \hline
0 & $\left]0,6\right[$ & $\left]0,12\right[$ & $\left]0,18\right[$ & $\left]0,24\right[$ \\ 
\hline 
1 & $\left]0.67,6\right[$ & $\left]0.36,12\right[$ & $\left]0.25,18\right[$ & $\left]0.19,24\right[$ \\ 
\hline 
2 & $\left]2.20,6\right[$ & $\left]1.35,12\right[$ & $\left]0.95,18\right[$ & $\left]0.73,24\right[$ \\ 
\hline 
3 & $\left]4.04,6\right[$ & $\left]2.75,12\right[$ & $\left]2.02,18\right[$ & $\left]1.58,24\right[$ \\ 
\hline 
4 &                    & $\left]4.39,12\right[$ & $\left]3.37,18\right[$ & $\left]2.70,24\right[$ \\ 
\hline 
5 &                    & $\left]6.19,12\right[$ & $\left]4.91,18\right[$ & $\left]4.02,24\right[$ \\ 
\hline 
6 &                    & $\left]8.07,12\right[$ & $\left]6.59,18\right[$ & $\left]5.49,24\right[$ \\ 
\hline 
7 &                    & $\left]10.02,12\right[$ & $\left]8.36,18\right[$ & $\left]7.09,24\right[$ \\ 
\hline 
8 &                    &                        & $\left]10.21,18\right[$ & $\left]8.78,24\right[$ \\ 
\hline 
9 &                    &                        & $\left]12.11,18\right[$ & $\left]10.55,24\right[$ \\ 
\hline 
10 &                    &                       & $\left]14.04,18\right[$ & $\left]12.37,24\right[$ \\ 
\hline 
11 &                    &                       & $\left]16.01,18\right[$ & $\left]14.24,24\right[$ \\ 
\hline 
\end{tabular} 
\end{center}
The lengths of these intervals shrink to zero as $m$ is chosen to be as large as possible, i.e. $m \nearrow 4l$. We also know by Lemma \ref{lem-irredcase} that the reducibility of the definiteness polynomial $P$ corresponds to the place of $n$ in the chain
\begin{equation} \label{bounds-red}
0 \ \leq \ \dfrac{3}{2}\sqrt{4l^2+2m^2}-3l \ < \ n_{\textrm{irred}}\ < \ \dfrac{8l^2+m^2}{4l} \ \leq \ n_{\textrm{red}} \ < \ 6l,
\end{equation}
the case $n_{\textrm{irred}}$ meaning $P$ is irreducible, and $n_{\textrm{red}}$ that $P$ is reducible. Note that, if $l$ and $m$ are given, the probability of a randomly chosen Finsler metric (value of $n$) yielding the reducible case depends only on the ratio of $m$ and $l$ because
\begin{align*}
&\dfrac{6l-\dfrac{8l^2+m^2}{4l}}{6l-\left(\dfrac{3}{2}\sqrt{4l^2+2m^2}-3l \right)} = \dfrac{\dfrac{2}{3}\,\dfrac{24l^2-8l^2-m^2}{4l}}{4l-\sqrt{4l^2+2m^2}+2l} \\
&= \dfrac{1}{6l} \, \dfrac{16l^2-m^2}{6l-\sqrt{4l^2+2m^2}} \, \dfrac{6l+\sqrt{4l^2+2m^2}}{6l+\sqrt{4l^2+2m^2}} \\
&=\dfrac{1}{6l} \, \dfrac{16l^2-m^2}{32l^2-2m^2} \left(6l+\sqrt{4l^2+2m^2}\right) \\
&= \dfrac{1}{2}+\dfrac{1}{12}\sqrt{4+2\left(\dfrac{m}{l} \right)^2}.
\end{align*}
Keeping the value of $m$ fixed and increasing the value of $l\to \infty$, this probability tends to $2/3$, i.e. for large values, the reducible case is about twice as likely. The critical value
$$n=\dfrac{8l^2+m^2}{4l}$$
corresponds the Riemannian case because $A(x,y)=F^4(x,y)$ is the square of a quadratic polynomial:
$$A(x,y)=l(x)\left(y_1^2+y_2^2+\dfrac{m(x)}{2l(x)} y_1 y_2\right)^2.$$

So far, we mainly looked at examples where the coefficients were constants (basically, we worked at one tangent space at a fixed point $p\in M$ of the base manifold). In general, they can depend on the position coordinates. To extend the metrics we used a theoretical  continuity argument but we are going to give an explicit example for a locally symmetric polynomial metric with non-constant coefficients as well. 

\begin{figure}
\includegraphics[width=.5\linewidth]{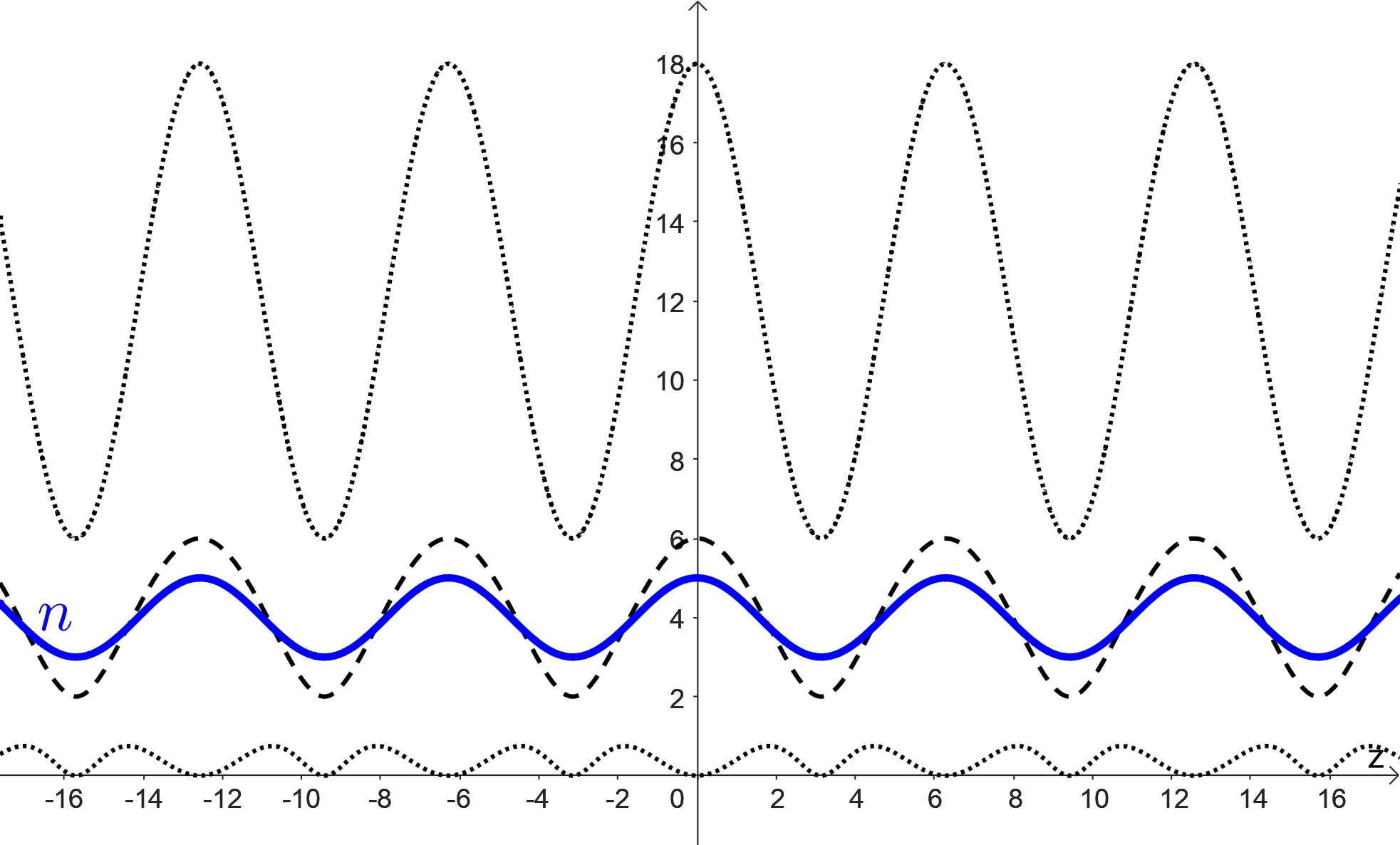}
\caption{Example \ref{sincosex2d}: the graph of the coefficient function $n$ (bold curve) with the lower and upper bounds $\frac{3}{2}\sqrt{4l^2+2m^2}-3l$ and $6l$, respectively (dotted curves), and the critical values $\frac{8l^2+m^2}{4l}$ (dashed curve) as functions of $z:=x_1x_2$. The reducibility of the definiteness polynomial changes each time $n$ crosses the dashed curve.}
\end{figure}

\begin{Exm} \label{sincosex2d} {\emph{Let's set $l(x)=\cos(x_1x_2)+2$, $m(x)=\sqrt{2}\sin(x_1x_2)$ and $n(x)=\cos(x_1x_2)+4$. It is a direct calculation to verify that we have a (positive definite) Finsler metric because
\[ \dfrac{3}{2}\sqrt{4l^2+2m^2}-3l < 1 < 3 \leq n \leq 5 < 6 \leq 6l. \]
Only the first inequality requires some work, but it can be proved indirectly (writing $z:=x_1x_2$) by
\begin{gather*}
 \dfrac{3}{2}\sqrt{4(\cos(z)+2)^2+2(\sqrt{2}\sin(z))^2}-3(\cos(z)+2) > 1 \\
 3\sqrt{5+4\cos(z)} > 7+3\cos(z) \\
 45+36\cos(z) > 49+42\cos(z)+9\cos^2(z) \\
 0 >4+6\cos(z)+9\cos(z)^2=(1+3\cos(z))^2+3,
\end{gather*}
which is a contradiction. It can also be checked (by a computing program, for example) that the reducibility of the definiteness polynomial can change from point to point.}}
\end{Exm}

\section{Computations in 3D}

If the base manifold is of dimension $3$, then the fourth power of a locally symmetric fourth root metric must be of the form 
\[P(s^1, s^2, s^3)=F^4(x,y)=A(x,y)=a(x) (y_1+y_2+y_3)^4 + b(x) (y_1+y_2+y_3)^2 ( y_1 y_2 + y_1 y_3 + y_2 y_3 ) +$$
$$ c(x) ( y_1 y_2 + y_1 y_3 + y_2 y_3 )^2 +d(x) (y_1+y_2+y_3)y_1 y_2 y_3,\]
where $a(x)=c_{400}(x)$, $b(x)=c_{210}(x)$, $c(x)=c_{020}(x)$ and $d(x)=c_{101}(x)$. Introducing the functions
\[l(x):=a(x), \ m(x):= 4a(x)+b(x), \ n(x):= 6a(x) + 2b(x) +c(x),\]
\[q(x) := 12 a(x)+ 5b(x) + 2c(x)+d(x),\]
we have that 
\begin{gather}
A(x,y) = l \left(y_1^4+y_2^4+y_3^4 \right)+m\left(y_1^3 y_2 + y_1^3 y_3 +  y_2^3 y_1 + y_2^3 y_3 + y_3^3 y_1 + y_3^3 y_2  \right)+  \label{3dmetric} \\
n\left(y_1^2y_2^2 + y_1^2y_3^2 + y_2^2y_3^2\right)+ q\left(y_1^2y_2y_3+y_1 y_2^2 y_3+ y_1 y_2 y_3^2 \right). \notag
\end{gather}
The partial derivatives of $A$ are
\begin{align*}
A_1(x,y) =\ & 4ly_1^3+ m\left(3y_1^2 y_2+3y_1^2 y_3+y_2^3+y_3^3\right)+ \\
& n\left(2y_1 y_2^2+2y_1y_3^2\right)+  q\left(2y_1y_2y_3+y_2^2y_3+y_2y_3^2\right), \\
A_2(x,y) =\ & 4ly_2^3+m\left(y_1^3+3y_1y_2^2+3y_2^2y_3+y_3^3\right)+ \\
& n\left(2y_1^2y_2+2y_2y_3^2\right)+q\left(y_1^2y_3+2y_1y_2y_3+y_1y_3^2\right), \\
A_3(x,y) =\ & 4ly_3^3+m\left(y_1^3+3y_1y_3^2+y_2^3+3y_2y_3^2\right)+ \\ & n\left(2y_1^2 y_3+2y_2^2y_3\right)+q\left(y_1^2y_2+y_1y_2^2+2y_1y_2y_3\right), \\[3pt]
A_{11}(x,y)=\ &  12l y_1^2 + m\left( 6y_1y_2+6y_1y_3\right)+ n \left(2y_2^2 + 2y_3^2\right) + 2q y_2y_3,\\
A_{12}(x,y)=A_{21}(x,y) =\ & m \left(3y_1^2 + 3y_2^2\right)+ 4n y_1y_2 + q\left(2y_1y_3+ 2y_2y_3 + y_3^2\right),\\
A_{13}(x,y)=A_{31}(x,y) =\ &  m \left(3y_1^2+3y_3^2 \right)+4n y_1y_3+ q \left( 2y_1y_2 + y_2^2 + 2y_2y_3\right),\\
A_{22}(x,y) =\ & 12 l y_2^2 + m \left(6 y_1y_2+ 6 y_2 y_3\right)+ n\left(2 y_1^2 + 2 y_3^2\right)+ 2q  y_1y_3,\\
A_{23}(x,y)=A_{32}(x,y) =\ & m \left( 3y_2^2 + 3y_3^2\right)+ 4n y_2y_3  + q\left(y_1^2 + 2y_1y_2 + 2y_1y_3\right),\\
A_{33}(x,y) =\ & 12 l y_3^2 + m \left(6 y_1y_3+ 6 y_2y_3\right)+ n \left(2 y_1^2 + 2 y_2^2\right)+ 2q y_1y_2.
\end{align*}
The Hessian 
\[ A_{ij}=\left[
\begin{matrix}
A_{11} & A_{12} & A_{13}\\ A_{21} & A_{22} & A_{23} \\ A_{31} & A_{32} & A_{33}
\end{matrix}
\right] \]
is of the form
\[A_{11}=
\left[
\begin{matrix}
y_1 & y_2 & y_3 
\end{matrix}
\right] 
\left[
\begin{matrix}
12 l & 3m & 3m \\  3m & 2n & q \\  3m & q & 2n
\end{matrix}
\right]
\left[
\begin{matrix}
y_1 \\  y_2 \\ y_3
\end{matrix}
\right], \ A_{12}=A_{21}=
\left[
\begin{matrix}
y_1 & y_2 & y_3 
\end{matrix}
\right]
\left[
\begin{matrix}
3m & 2n & q \\  2n & 3m & q \\  q & q & q
\end{matrix}
\right]
\left[
\begin{matrix}
y_1 \\  y_2 \\ y_3
\end{matrix}
\right],$$
$$A_{22}=
\left[
\begin{matrix}
y_1 & y_2 & y_3 
\end{matrix}
\right]
\left[
\begin{matrix}
2n & 3m & q \\  3m & 12l & 3m \\  q & 3m & 2n
\end{matrix}
\right]
\left[
\begin{matrix}
y_1 \\  y_2 \\ y_3
\end{matrix}
\right], \ A_{13}=A_{31}=
\left[
\begin{matrix}
y_1 & y_2 & y_3 
\end{matrix}
\right]
\left[
\begin{matrix}
3m & q & 2n \\  q & q & q \\  2n & q & 3m
\end{matrix}
\right]
\left[
\begin{matrix}
y_1 \\  y_2 \\ y_3
\end{matrix}
\right],$$
$$ A_{33}=
\left[
\begin{matrix}
y_1 & y_2 & y_3 
\end{matrix}
\right]
\left[
\begin{matrix}
2n & q & 3m \\  q & 2n & 3m \\  3m & 3m & 12l
\end{matrix}
\right]
\left[
\begin{matrix}
y_1 \\  y_2 \\ y_3
\end{matrix}
\right], \ A_{23}=A_{32}= 
\left[
\begin{matrix}
y_1 & y_2 & y_3 
\end{matrix}
\right]
\left[
\begin{matrix}
q & q & q \\  q & 3m & 2n \\  q & 2n & 3m
\end{matrix}
\right]
\left[
\begin{matrix}
y_1 \\  y_2 \\ y_3
\end{matrix}
\right]. \]

\begin{Lem} \label{lem1-3d} $A_{ij}$ is positive definite if and only if for all nonzero elements $v\in \pi^{-1}(U)\subseteq TM$,
\[ A_{11}>0, \hspace{1cm} A_{11}A_{22}-A_{12}^2>0 \hspace{1cm} \text{and} \hspace{1cm} \det A_{ij}>0.  \]
\end{Lem}

Though the setting is exactly the same as in 2D, the computations are incomparably more difficult. Using MAPLE, we get that
\begin{dmath*}
A_{11}A_{22}-A_{12}^2=\left( 24ln-9m^2 \right) y_1^4+ \left( 72lm-12mn
 \right) y_1^3y_2+ \left( 24lq+12mn-12mq \right) y_1
^3y_3+ \left( 144l^2+18m^2-12n^2 \right) y_1
^2y_2^2+ \left( 72lm+36m^2-12nq \right) y_1^2y
_2y_3+ \left( 24ln+6mq+4n^2-4q^2 \right) y_1
^2y_3^2+ \left( 72lm-12mn \right) y_1y_2^3+
 \left( 72lm+36m^2-12nq \right) y_1y_2^2y_3+
 \left( 36m^2+24mn-8nq-4q^2 \right) y_1y_2y_
3^2+ \left( 12mn+4nq-4q^2 \right) y_1y_3^3+
 \left( 24ln-9m^2 \right) y_2^4+ \left( 24lq+12mn-12
mq \right) y_2^3y_3+ \left( 24ln+6mq+4n^2-4q^
2 \right) y_2^2y_3^2+ \left( 12mn+4nq-4q^2
 \right) y_2y_3^3+ \left( 4n^2-q^2 \right) y_3^4
\end{dmath*}
and
\begin{dmath*}
\det A_{ij} = a \left( y_1^6+y_2^6+y_3^6 \right) +b \left( y_1
^5y_2+y_1^5y_3+y_2^5y_1+y_3^5y_1+y_2^
5y_3+y_3^5y_2 \right) +c \left( y_1^4y_2^2+y_
1^4y_3^2+y_2^4y_1^2+y_3^4y_1^2+y_2
^4y_3^2+y_3^4y_2^2 \right) +d \left( y_1^4
y_2y_3+y_1y_2^4y_3+y_1y_2y_3^4
 \right) +e \left( y_1^3y_2^3+y_1^3y_3^3+y_2
^3y_3^3 \right) +f \left( y_1^3y_2^2y_3+y_1
^3y_2y_3^2+y_1^2y_2^3y_3+y_1^2y_2
y_3^3+y_1y_2^3y_3^2+y_1y_2^2y_3^3
 \right) +gy_1^2y_2^2y_3^2,
\end{dmath*}
where
\begin{align*}
a =\ & 48ln^2-12lq^2-36m^2n+18m^2q \\
b =\ & 144lmn+48lnq-48lq^2-54m^3+18m^2q-24mn^2+6mq^2 \\
c =\ & 288l^2n-108lm^2+72lmq+48ln^2-48lq^2+54m^3 \\ & +18m^2n-54m^2q+12mnq+6mq^2-24n^3+6nq^2 \\
d =\ & 432lm^2+288lmn-96lnq-48lq^2-108m^3 \\ & -72m^2n+96mnq-24mq^2-24n^2q+6q^3 \\
e =\ & 288l^2q+288lmn-288lmq+72m^2n+36m^2q \\ & -48mn^2-24mq^2-24n^2q+24nq^2 \\
f =\ & 864l^2m+432lm^2-144lmn-144lnq+108m^3 \\ & +72m^2n-36m^2q+48mn^2-96mnq-24mq^2+24n^2q+12q^3 \\
g =\ & 1728l^3+648lm^2-432ln^2+540m^3-162m^2q-216mnq+144n^3+18q^3.
\end{align*}

\subsection{Necessary conditions} The positive definiteness of the matrix $A_{ij}$ implies some necessary conditions.

\begin{Lem} The condition $A_{11}>0$ is equivalent to
\begin{equation} \label{nec01in3d}
0 < l \hspace{1cm} \text{and} \hspace{1cm} \dfrac{3m^2-4ln}{2l} < q < 2n.
\end{equation}
\end{Lem}

\begin{proof} We can rephrase $A_{11}>0$ in terms of the corner minors as
\begin{itemize}
\item[($\Delta_1$)] $12l>0$ \ $\Longleftrightarrow$ \ $l>0$,
\item[($\Delta_2$)] $24ln-9m^2>0$ \ $\Longleftrightarrow$ \ $8ln>3m^2$,
\item[($\Delta_3$)] $48ln^2-12lq^2-36m^2n+18m^2q>0$.
\end{itemize}
Rearranging condition ($\Delta_3$) gives
\begin{align*}
8ln^2-2lq^2-6m^2n+3m^2q &> 0 \\
2l(4n^2-q^2)-3m^2(2n-q) &> 0 \\
2l(2n-q)(2n+q) &> 3m^2(2n-q).
\end{align*}
The choice $q=2n$ leads to a contradiction immediately; so does $q > 2n \Leftrightarrow 2n-q<0$, because ($\Delta_3$) and ($\Delta_2$)  give
\begin{align*}
2l(2n+q) &< 3m^2 < 8ln \\
2lq &< 4ln \\
q &< 2n.
\end{align*}
Therefore $q<2n \Leftrightarrow 2n-q>0$. After dividing inequality
$$2l(2n-q)(2n+q) > 3m^2(2n-q)$$
by $2n-q$ we have the lower bound of \eqref{nec01in3d} for the coefficient $q$. Conversely, $l>0$ is equivalent to ($\Delta_1$). \eqref{nec01in3d} also implies ($\Delta_2$) because
\[3m^2 < 2lq+4ln=2l(2n+q) < 2l(2n+2n) =8ln.  \]
Finally, ($\Delta_3$) can be given in the following way:
\[\begin{array}{rclcr}
2l(2n+q) &>& 3m^2 & \hspace{2cm} & /\cdot (2n-q)>0 \\[3pt]
2l(2n-q)(2n+q) &>& 3m^2(2n-q) &&   \\[3pt]
2l(4n^2-q^2) &>& 3m^2(2n-q) && \\[3pt]
2l(4n^2-q^2)-3m^2(2n-q) &>& 0 && \\[3pt]
8ln^2-2lq^2-6m^2n+3m^2q &>& 0. && 
\end{array} \] 
\end{proof}

\begin{Lem} If $A_{ij}$ is positive definite in 3D, then the coefficients $l,m,n$ must satisfy all the conditions from Theorem \ref{2dmaintheorem} in 2D.
\end{Lem}

\begin{proof} $A_{11}>0$ in 3D obviously implies the positiveness of $A_{11}$ in 2D (see conditions ($\Delta_1$) and ($\Delta_2$). Furthermore, taking condition $A_{11}A_{22}-A_{12}^2>0$ in 3D and setting $y_3=0$ gives
\[(24ln-9m^2) \left(y_1^4+y_2^4\right)+(72lm-12mn)\left(y_1^3y_2+y_1y_2^3 \right)+(144l^2+18m^2-12n^2)y_1^2y_2^2 > 0. \]
This is exactly the condition of $\det A_{ij}>0$ in 2D.
\end{proof}

To sum up, we currently have the necessary conditions 
\begin{equation} \label{3dneccond}
\dfrac{3}{2}\sqrt{4l^2+2m^2}-3l < n < 6l \quad \text{and} \quad \dfrac{3m^2-4ln}{2l} < q < 2n
\end{equation}
for the positive definiteness of $A_{ij}$ in 3D.

\begin{Con} {\emph{Find a complete system of conditions for positive definiteness in terms of the coefficients of a locally symmetric fourth root metric in 3D.}}
\end{Con}

\subsection{Computations under a special choice of the coefficients} 

\begin{Exm} {\emph{Let us choose $l=1, m=2, n=3$ and $q=4$. We are going to check that this yields a positive definite Hessian $A_{ij}$. The corner minors are
\[A_{11}=\left[
\begin{matrix}
y_1 & y_2 & y_3 
\end{matrix}
\right] 
\left[
\begin{matrix}
12 & 6 & 6 \\  6 & 6 & 4 \\  6 & 4 & 6
\end{matrix}
\right]
\left[
\begin{matrix}
y_1 \\  y_2 \\ y_3
\end{matrix}
\right], \]
\begin{dmath*}A_{11}A_{22}-A_{12}^2= 36y_1^4+72y_1^3y_2+72y_1^3y_3+108y_1^2y_2^2+144y_1^2y_2y_3+92y_1^2y_3^2+72y_1y_2^3+144y_1y_2^2y_3+128y_1y_2y_3^2+56y_1y_3^3+36y_2^4+72y_2^3
y_3+92y_2^2y_3^2+56y_2y_3^3+20y_3^4 \\
= 4(y_1^2+y_2^2+y_3^2+y_1y_2+y_1y_3+y_2y_3)(9y_1^2+9y_2^2+5y_3^2+9y_1y_2+9y_1y_3+9y_2y_3),
\end{dmath*}
\begin{dmath*}
\det A_{ij}= 96y_1^6+288y_1^5y_2+288y_1^5y_3+576y_1
^4y_2^2+864y_1^4y_2y_3+576y_1^4y_3
^2+672y_1^3y_2^3+1440y_1^3y_2^2y_3+
1440y_1^3y_2y_3^2+672y_1^3y_3^3+576
y_1^2y_2^4+1440y_1^2y_2^3y_3+2016y_1^
2y_2^2y_3^2+1440y_1^2y_2y_3^3+576y_
1^2y_3^4+288y_1y_2^5+864y_1y_2^4y_
3+1440y_1y_2^3y_3^2+1440y_1y_2^2y_3
^3+864y_1y_2y_3^4+288y_1y_3^5+96y_2
^6+288y_2^5y_3+576y_2^4y_3^2+672y_2^
3y_3^3+576y_2^2y_3^4+288y_2y_3^5+96
y_3^6 \\
=96(y_1^2+y_2^2+y_3^2+y_1y_2+y_1y_3+y_2y_3)^3.
\end{dmath*}
$A_{11}$ is always positive, because the corner minors of the base matrix are $12, 36$ and $96$, respectively. Conditions
$$A_{11}A_{22}-A_{12}^2 > 0\quad \textrm{and} \quad \det A_{ij} >0$$
are also satisfied (except the origin) because the factors are positive definite quadratic forms. Using a continuity argument we also have examples with non-constant coefficients by choosing $l(x)\approx 1$, $m(x)\approx 2$, $n(x)\approx 3$ and $q(x)\approx 4$.}}
\end{Exm}

\section{Acknowledgments}

M\'{a}rk Ol\'{a}h has received funding from the HUN-REN Hungarian Research Network.

\end{document}